\documentclass[11pt,twoside]{article}
\usepackage{mathrsfs}
\usepackage{amsmath}
\usepackage{amsthm}
\input amssymb.sty
\usepackage{amsfonts}
\usepackage{amssymb}
\usepackage{latexsym}
\usepackage[all]{xy}

\date{\empty}
\allowdisplaybreaks[4] \footskip=15pt

\numberwithin{equation}{section} \theoremstyle{plain}
\newtheorem*{thm*}{Main Theorem}
\newtheorem{theorem}{Theorem}[section]
\newtheorem{corollary}[theorem]{Corollary}
\newtheorem*{corollary*}{Corollary}

\newtheorem*{claim*}{Claim}
\newtheorem{lemma}[theorem]{Lemma}
\newtheorem*{lemma*}{Lemma}

\newtheorem*{proposition*}{Proposition}

\newtheorem*{remark*}{Remark}

\newtheorem*{example*}{Example}

\newtheorem*{question*}{Question}
\newtheorem{definition}[theorem]{Definition}
\newtheorem*{definition*}{Definition}

\newtheorem*{acknowledgements*}{ACKNOWLEDGEMENTS}

\begin{document}
\begin{center}
{\large  \bf The Moore-Penrose inverse in rings with involution}\\
\vspace{0.8cm} {\small \bf Sanzhang  Xu and Jianlong Chen}
\footnote{Corresponding author.
E-mail: jlchen@seu.edu.cn
E-mail: xusanzhang5222@126.com}

\vspace{0.6cm} {\rm Department of Mathematics, Southeast University, Nanjing 210096, China}
\end{center}

\bigskip

{ \bf  Abstract:}  \leftskip0truemm\rightskip0truemm
Let $R$ be a unital ring with involution.
In this paper, several new necessary and sufficient conditions for the existence of the Moore-Penrose inverse of an element in a ring $R$ are given.
In addition, the formulae of the Moore-Penrose inverse of an element in a ring are presented.


 \bigskip

\section { \bf Introduction}

Let $R$ be a $*$-ring, that is a ring with an involution $a\mapsto a^*$
satisfying $(a^*)^*=a$, $(ab)^*=b^*a^*$ and $(a+b)^*=a^*+b^*$.
We say that $b\in R$ is the Moore-Penrose inverse of $a\in R$, if the following hold:
$$aba=a, \quad bab=b, \quad (ab)^{\ast}=ab \quad (ba)^{\ast}=ba.$$
There is at most one $b$ such that above four equations hold.
If such an element $b$ exists, it is denoted by $a^{\dagger}$. The set of all Moore-Penrose invertible elements will be denoted by $R^{\dagger}$.
An element $b\in R$ is an inner inverse of $a\in R$ if $aba=a$ holds. The set of all inner inverses of $a$ will be denoted by $a\{1\}$.
An element $a\in R$ is said to be group invertible
if there exists $b\in R$ such that the following equations hold:
$$aba=a, \quad bab=b, \quad ab=ba.$$
The element $b$ which satisfies the above equations is called a group inverse of $a$.
If such an element $b$ exists, it is unique and denoted by $a^\#$. The set of all group invertible elements will be denoted by $R^\#$.

An element $a\in R$ is called an idempotent if $a^{2}=a.$
$a$ is called a projection if $a^{2}=a=a^{\ast}.$
$a$ is called normal if $aa^{\ast}=a^{\ast}a$.
$a$ is called a Hermite element if $a^{\ast}=a$.
$a$ is said to be an EP element if $a\in R^{\dagger}\cap R^\#$ and $a^{\dagger}=a^\#.$ The set of all EP elements will be denoted by $R^{EP}$.
$\tilde{a}$ is called a $\{1,3\}$-inverse of $a$ if we have $a\tilde{a}a=a,~(a\tilde{a})^{\ast}=a\tilde{a}$. The set of all $\{1,3\}$-invertible elements will be denoted by $R^{\{1,3\}}$. Similarly, an element $\hat{a}\in R$  is called a $\{1,4\}$-inverse of $a$ if $a\hat{a}a=a,~(\hat{a}a)^{\ast}=\hat{a}a$. The set of all $\{1,4\}$-invertible elements will be denoted by $R^{\{1,4\}}$.

We will also use the following notations: $aR=\{ax\mid x\in R\}$, $Ra=\{xa\mid x\in R\}$, $^{\circ}a=\{x\in R\mid xa=0\}$ and $a^{\circ}=\{x\in R\mid ax=0\}$.

In \cite{C,MDA,MDB,TW}, the authors showed that the equivalent conditions such that $a\in R$ to be an EP element are closely related with
powers of the group and Moore-Penrose inverse of $a$.
 Motivated by the above statements, in this paper, we will show that the existence of the Moore-Penrose inverse of an element in a ring $R$
is closely related with powers of some Hermite elements, idempotents and projections.

Recently, Zhu, Chen and Patri\'{c}io in \cite{CZ} introduced the concepts of left $\ast$-regular and right $\ast$-regular.
We call an element $a\in R$ is left (right) $\ast$-regular if there exists $x\in R$ such that $a=aa^{\ast}ax$ $(a=xaa^{\ast}a)$.
They proved that $a\in R^{\dagger}$ if and only if $a$ is left $\ast$-regular if and only if $a$ is right $\ast$-regular. With the help
of left (right) $\ast$-regular, we will give more equivalent conditions
of the Moore-Penrose of an element in a ring.

In \cite{H}, Hartwig proved that for an element $a\in R$, $a$ is $\{1,3\}$-invertible with $\{1,3\}$-inverse $x$ if and only if $x^{\ast}a^{\ast}a=a$
and it also proved that $a$ is $\{1,4\}$-invertible with $\{1,4\}$-inverse $y$ if and only if $aa^{\ast}y^{\ast}=a.$
In \cite{P}, Penrose proved the following result in the complex matrix case, yet it is true for an element in a ring with involution, $a\in R^{\dagger}$ if and only if $a\in Ra^{\ast}a\cap aa^{\ast}R$.
In this case, $a^{\dagger}=y^{\ast}ax^{\ast}$, where $a=aa^{\ast}y=xa^{\ast}a.$

It is well-known that an important feature of the Moore-Penrose inverse is that it can be used to represent projections.
Let $a\in R^{\dagger}$, then if we let $p=aa^{\dagger}$ and $q=a^{\dagger}a$, we have
$p$ and $q$ are projections. In \cite{HC}, Han and Chen proved that
$a\in R^{\{1,3\}}$ if and only if there exists unique projection $p\in R$ such that $aR=pR.$
And, it is also proved that
$a\in R^{\{1,4\}}$ if and only if there exists unique projection $q\in R$ such that $Ra=Rq.$
We will show that the existence of the Moore-Penrose inverse
is closely related with some Hermite elements and projections.

In \cite[Theorem 2.4]{K}, Koliha proved that $a\in \mathcal{A}^{\dagger}$ if and only if $a$ is well-supported, where $\mathcal{A}$ is a C$^{\ast}$-algebra.
In \cite[Theorem 1]{KDC}, Koliha, Djordjevi\'{c} and Cvetkvi\'{c} proved that $a\in R^{\dagger}$ if and only if $a$ is left $\ast$-cancellable and well-supported. Where an element
$a\in R$ is called well-supported if there exists projection $p\in R$ such that $ap=a$ and $a^{\ast}a+1-p\in R^{-1}$. In Theorem \ref{mp-in-ring-semigroup5}, we will show that the condition that $a$ is left $\ast$-cancellable in \cite[Theorem 1]{KDC} can be dropped. Moreover, we prove that $a\in R^{\dagger}$ if and only if there exists $e^{2}=e \in R$ such that $ea=0$ and $aa^{\ast}+e$ is left invertible. And, it is also proved that $a\in R^{\dagger}$ if and only if there exists $b \in R$ such that $ba=0$ and $aa^{\ast}+b$ is left invertible.

In \cite{H}, Hartwig proved that $a\in R^{\{1,3\}}$ if and only if $R=aR\oplus (a^{\ast})^{\circ}$. And, it is also proved that
$a\in R^{\{1,4\}}$ if and only if $R=Ra\oplus ^{\circ}\!(a^{\ast})$.
Hence $a\in R^{\dagger}$ if and only if $R=aR\oplus (a^{\ast})^{\circ}=Ra\oplus ^{\circ}\!(a^{\ast})$.
We will show that $a\in R^{\dagger}$ if and only if
$R=a^{\circ}\oplus (a^{\ast}a)^{n}R$. It is also show that $a\in R^{\dagger}$ if and only if $R=a^{\circ}+ (a^{\ast}a)^{n}R$
, for all choices $n\in \mathbf{N^{+}}$, where $\mathbf{N^{+}}$ stands for the set of all positive integers.
\section { \bf Preliminaries}

In this section, several auxiliary lemmas are presented.

\begin{lemma} \cite{H}\label{13-14-inverse}
Let $a\in R$. Then we have the following results:\\
$(1)$ $a$ is $\{1,3\}$-invertible with $\{1,3\}$-inverse $x$ if and only if $x^{\ast}a^{\ast}a=a;$\\
$(2)$ $a$ is $\{1,4\}$-invertible with $\{1,4\}$-inverse $y$ if and only if $aa^{\ast}y^{\ast}=a.$
\end{lemma}

The following two Lemmas can be found in \cite{P} in the complex matrix case, yet it is true for an element in a ring with involution.

\begin{lemma} \label{mp-1314}
Let $a\in R$. Then $a\in R^{\dagger}$ if and only if there exist $x,y\in R$ such that $x^{\ast}a^{\ast}a=a$ and $aa^{\ast}y^{\ast}=a.$ In this case,
$a^{\dagger}=yax.$
\end{lemma}

\begin{lemma} \label{normal-ep}
Let $a\in R^{\dagger}$. Then:\\
$(1)$ $aa^{\ast}, a^{\ast}a\in R^{EP}$ and $(aa^{\ast})^{\dagger}=(a^{\ast})^{\dagger}a^{\dagger}$ and $(a^{\ast}a)^{\dagger}=a^{\dagger}(a^{\ast})^{\dagger}$;\\
$(2)$ If $a$ is normal, then $a\in R^{EP}$ and $(a^{k})^{\dagger}=(a^{\dagger})^{k}$ for $k\in \mathbf{N^{+}}$.
\end{lemma}

We will give a generalization of Lemma \ref{normal-ep}(1) in the following lemma.
\begin{lemma} \label{mp-star-ep}
Let $a\in R^{\dagger}$ and $n, m\in \mathbf{N^+}$. Then $(aa^{\ast})^{n}, (a^{\ast}a)^{m}\in R^{EP}$.
\end{lemma}
\begin{proof}
Suppose $a\in R^{\dagger}$, by Lemma \ref{normal-ep} and $(aa^{\ast})^{\ast}=aa^{\ast}$, we have
$((aa^{\ast})^{n})^{\dagger}=((aa^{\ast})^{\dagger})^{n}$, $((a^{\ast}a)^{n})^{\dagger}=((a^{\ast}a)^{\dagger})^{n},$
$(aa^{\ast})^{\dagger}=(a^{\ast})^{\dagger}a^{\dagger}$, $(a^{\ast}a)^{\dagger}=a^{\dagger}(a^{\ast})^{\dagger},$
$aa^{\ast}(aa^{\ast})^{\dagger}=(aa^{\ast})^{\dagger}aa^{\ast}$ and $a^{\ast}a(a^{\ast}a)^{\dagger}=(a^{\ast}a)^{\dagger}a^{\ast}a.$
Thus we have
\begin{enumerate}
  \item [$(i)$] $(aa^{\ast})^{n}((aa^{\ast})^{n})^{\dagger}(aa^{\ast})^{n}
=(aa^{\ast})^{n}((aa^{\ast})^{\dagger})^{n}(aa^{\ast})^{n}
=(aa^{\ast}(a^{\ast})^{\dagger}a^{\dagger}aa^{\ast})^{n}
=(aa^{\ast})^{n};$
  \item [$(ii)$] $((aa^{\ast})^{n})^{\dagger}(aa^{\ast})^{n}((aa^{\ast})^{n})^{\dagger}
=((aa^{\ast})^{\dagger})^{n}(aa^{\ast})^{n}((aa^{\ast})^{\dagger})^{n}
=((aa^{\ast})^{n})^{\dagger};$
  \item [$(iii)$] $[(aa^{\ast})^{n}((aa^{\ast})^{n})^{\dagger}]^{\ast}
=[(aa^{\ast})^{n}((aa^{\ast})^{\dagger})^{n}]^{\ast}
=[(aa^{\ast}(aa^{\ast})^{\dagger})^{n}]^{\ast}
=(aa^{\ast})^{n}((aa^{\ast})^{n})^{\dagger};$
 \item [$(iv)$] $[((aa^{\ast})^{n})^{\dagger}(aa^{\ast})^{n}]^{\ast}
=[((aa^{\ast})^{\dagger})^{n}(aa^{\ast})^{n}]^{\ast}
=[((aa^{\ast})^{\dagger}aa^{\ast})^{n}]^{\ast}
=((aa^{\ast})^{n})^{\dagger}(aa^{\ast})^{n};$
 \item [$(v)$] $(aa^{\ast})^{n}((aa^{\ast})^{n})^{\dagger}
=(aa^{\ast})^{n}((aa^{\ast})^{\dagger})^{n}
=(aa^{\ast}(aa^{\ast})^{\dagger})^{n}
=((aa^{\ast})^{\dagger}aa^{\ast})^{n}
=((aa^{\ast})^{n})^{\dagger}(aa^{\ast})^{n}.$
\end{enumerate}

By the definition of the EP element, we have $(aa^{\ast})^{n}\in R^{EP}$.
Similarly, $(a^{\ast}a)^{m}\in R^{EP}$.
\end{proof}

\begin{definition}\label{cancell}
An element $a\in R$ is $\ast$-cancellable if
$a^{\ast}ax=0$ implies $ax=0$ and $yaa^{\ast}=0$ implies $ya=0$.
\end{definition}

The equivalence of conditions $(1)$, $(3)$ and $(5)$ in the following lemma was also proved by Puystjens and Robinson \cite[Lemma 3]{PR} in categories with involution.
\begin{lemma} \cite{KP}\label{star-mp}
Let $a\in R$. Then the following conditions are equivalent:\\
$(1)$ $a\in R^{\dagger}$;\\
$(2)$ $a^{\ast}\in R^{\dagger}$;\\
$(3)$ $a$ is $\ast$-cancellable and $aa^{\ast}$ and $a^{\ast}a$ are regular;\\
$(4)$ $a$ is $\ast$-cancellable and $a^{\ast}aa^{\ast}$ is regular;\\
$(5)$ $a\in Ra^{\ast}a\cap aa^{\ast}R$.
\end{lemma}

\begin{lemma}\label{n-regular}
Let $a\in R^{\dagger}$ and $n,m\in \mathbf{N^+}$. Then\\
$(1)$ $(aa^{\ast})^{n}((aa^{\ast})^{n})^{\dagger}a=a$;\\
$(2)$ $a((a^{\ast}a)^{m})^{\dagger}(a^{\ast}a)^{m}=a$.
\end{lemma}
\begin{proof}
$(1)$ If $n=1$ and $aa^{\ast}(aa^{\ast})^{\dagger}aa^{\ast}=aa^{\ast}$, by $a$ is $\ast$-cancellable, we have
$aa^{\ast}(aa^{\ast})^{\dagger}a=a.$
Suppose if $n=k$, we have
$(aa^{\ast})^{k}((aa^{\ast})^{k})^{\dagger}a=a.$
By Lemma \ref{normal-ep}, we have
\begin{equation*}
\begin{split}\label{n-regular77}
&(aa^{\ast})^{k+1}[(aa^{\ast})^{k+1}]^{\dagger}a\\
&~=aa^{\ast}(aa^{\ast})^{k}[(aa^{\ast})^{\dagger}]^{k+1}a
=aa^{\ast}(aa^{\ast})^{k}[(aa^{\ast})^{\dagger}]^{k}(aa^{\ast})^{\dagger}a\\
&~=aa^{\ast}(aa^{\ast})^{k}[(aa^{\ast})^{\dagger}]^{k}(a^{\ast})^{\dagger}a^{\dagger}a
=aa^{\ast}(aa^{\ast})^{k}[(aa^{\ast})^{\dagger}]^{k}(a^{\dagger})^{\ast}a^{\dagger}a\\
&~=aa^{\ast}(aa^{\ast})^{k}[(aa^{\ast})^{\dagger}]^{k}(a^{\dagger}aa^{\dagger})^{\ast}a^{\dagger}a
=aa^{\ast}(aa^{\ast})^{k}[(aa^{\ast})^{\dagger}]^{k}aa^{\dagger}(a^{\dagger})^{\ast}a^{\dagger}a\\
&~=aa^{\ast}aa^{\dagger}(a^{\dagger})^{\ast}a^{\dagger}a
=aa^{\ast}(aa^{\dagger})^{\ast}(a^{\dagger})^{\ast}(a^{\dagger}a)^{\ast}\\
&~=a(aa^{\dagger}a)^{\ast}(a^{\dagger}aa^{\dagger})^{\ast}
=aa^{\ast}(a^{\dagger})^{\ast}\\
&~=a(a^{\dagger}a)^{\ast}
=aa^{\dagger}a
=a.
\end{split}
\end{equation*}
Thus by mathematical induction, we have $(aa^{\ast})^{n}((aa^{\ast})^{n})^{\dagger}a=a$.

$(2)$ It is similar to $(1)$.
\end{proof}

\begin{lemma}\cite{CZ}\label{aa-star-a}
Let $a\in R$. The following conditions are equivalent:\\
$(1)$ $a\in R^{\dagger}$;\\
$(2)$ $a\in aa^{\ast}aR$;\\
$(3)$ $a\in Raa^{\ast}a$.

In this case, $a^{\dagger}=(ax)^{\ast}axa^{\ast}=a^{\ast}ya(ya)^{\ast}$, where $a=aa^{\ast}ax=yaa^{\ast}a.$
\end{lemma}

\begin{lemma} \cite{PP}\label{eppp}
Let $a\in R.$ If $aR=a^\ast R$, then the following are equivalent:\\
$(1)$ $a\in R^{EP}$;\\
$(2)$ $a\in R^{\dagger}$;\\
$(3)$ $a\in R^\#.$
\end{lemma}
\section { \bf  Main results}

In this section, several necessary and sufficient conditions for the existence of the Moore-Penrose inverse of an element in a ring $R$ are given.

\begin{theorem}\label{mp-in-ring-semigroup1}
Let $a\in R$ and $m,n\in \mathbf{N^+}$. Then the following conditions are equivalent:\\
$(1)$ $a\in R^{\dagger}$;\\
$(2)$ $a\in R(a^{\ast}a)^{m}\cap (aa^{\ast})^{n}R$;\\
$(3)$ $a\in a(a^{\ast}a)^{n}R$;\\
$(4)$ $a\in R(aa^{\ast})^{n}a$;\\
$(5)$ $(aa^{\ast})^{n}\in R^{\dagger}$ and $(aa^{\ast})^{n}[(aa^{\ast})^{n}]^{\dagger}a=a$;\\
$(6)$ $(a^{\ast}a)^{n}\in R^{\dagger}$ and $a[(a^{\ast}a)^{n}]^{\dagger}(a^{\ast}a)^{n}=a$;\\
$(7)$ $a$ is $\ast$-cancellable and $(aa^{\ast})^{m}$ and $(a^{\ast}a)^{n}$ are regular;\\
$(8)$ $a$ is $\ast$-cancellable and $(a^{\ast}a)^{n}a^{\ast}$ is regular;\\
$(9)$ $a$ is $\ast$-cancellable and $a^{\ast}(aa^{\ast})^{n}$ is regular;\\
$(10)$ $a$ is $\ast$-cancellable and $(aa^{\ast})^{n}\in R^\#$;\\
$(11)$ $a$ is $\ast$-cancellable and $(a^{\ast}a)^{n}\in R^\#$;\\
$(12)$ $a$ is $\ast$-cancellable and $(aa^{\ast})^{n}\in R^{\dagger}$;\\
$(13)$ $a$ is $\ast$-cancellable and $(a^{\ast}a)^{n}\in R^{\dagger}$.\\
In this case, \\
$a^{\dagger}=y_{1}^{\ast}(aa^{\ast})^{m+n-2}ax_{1}^{\ast}
=x_{2}^{\ast}(a^{\ast}a)^{2n-1}x_{2}a^{\ast}
=a^{\ast}y_{2}(aa^{\ast})^{2n-1}y_{2}^{\ast}$,
where $a=x_{1}(a^{\ast}a)^{m}$, $a=(aa^{\ast})^{n}y_{1}$,
$a=a(a^{\ast}a)^{n}x_{2}$, $a=y_{2}(aa^{\ast})^{n}a$.
\end{theorem}
\begin{proof}
$(1)\Rightarrow(2)$ By Lemma \ref{n-regular} we can get
$(aa^{\ast})^{n}((aa^{\ast})^{n})^{\dagger}a=a$
and
$a((a^{\ast}a)^{m})^{\dagger}(a^{\ast}a)^{m}=a.$
Thus we have
$a\in R(a^{\ast}a)^{m}\cap (aa^{\ast})^{n}R$.

$(2)\Rightarrow(1)$ Suppose $a\in R(a^{\ast}a)^{m}\cap (aa^{\ast})^{n}R$, then for some $x_{1},y_{1}\in R$, we have
\begin{eqnarray}\label{mp3}
a=x_{1}(a^{\ast}a)^{m}~~\text{and}~~a=(aa^{\ast})^{n}y_{1}.
\end{eqnarray}
By (\ref{mp3}) and Lemma \ref{13-14-inverse}, we have
\begin{eqnarray}\label{mp4}
[x_{1}(a^{\ast}a)^{m-1}]^{\ast}\in a\{1,3\}~~\text{and}~~[(aa^{\ast})^{n-1}y_{1}]^{\ast}\in a\{1,4\}.
\end{eqnarray}
Thus by (\ref{mp4}) and Lemma \ref{mp-1314}, we have $a\in R^{\dagger}$ and
\begin{equation*} \label{nnmp5}
\begin{split}
a^{\dagger}
&=a^{(1,4)}aa^{(1,3)}
=[(aa^{\ast})^{n-1}y_{1}]^{\ast}a[x_{1}(a^{\ast}a)^{m-1}]^{\ast}\\
&=y_{1}^{\ast}(aa^{\ast})^{n-1}a(a^{\ast}a)^{m-1}x_{1}^{\ast}
=y_{1}^{\ast}(aa^{\ast})^{m+n-2}ax_{1}^{\ast}.
\end{split}
\end{equation*}

$(1)\Rightarrow(3)$ By Lemma \ref{normal-ep}, we have
$a^{\ast}a=a^{\ast}aa^{\dagger}(a^{\dagger})^{\ast}a^{\ast}a$ and
$a^{\dagger}(a^{\dagger})^{\ast}a^{\ast}a=a^{\ast}aa^{\dagger}(a^{\dagger})^{\ast}.$
Thus
\begin{equation*} \label{nnmp8}
\begin{split}
a
&~=aa^{\dagger}a=(aa^{\dagger})^{\ast}a=(a^{\dagger})^{\ast}a^{\ast}a
=(a^{\dagger})^{\ast}a^{\ast}aa^{\dagger}(a^{\dagger})^{\ast}a^{\ast}a
=(a^{\dagger})^{\ast}(a^{\ast}a)^{2}a^{\dagger}(a^{\dagger})^{\ast}\\
&~=((a^{\dagger})^{\ast}a^{\ast}a)a^{\ast}aa^{\dagger}(a^{\dagger})^{\ast}
~=(aa^{\dagger}a)a^{\ast}aa^{\dagger}(a^{\dagger})^{\ast}\\
&~=aa^{\ast}aa^{\dagger}(a^{\dagger})^{\ast}
=a(a^{\ast}aa^{\dagger}(a^{\dagger})^{\ast}a^{\ast}a)a^{\dagger}(a^{\dagger})^{\ast}
=a(a^{\ast}a)^{2}(a^{\dagger}(a^{\dagger})^{\ast})^{2}\\
&~=\cdots\cdots\\
&~=a(a^{\ast}a)^{n}(a^{\dagger}(a^{\dagger})^{\ast})^{n}.
\end{split}
\end{equation*}
Hence $a\in a(a^{\ast}a)^{n}R$.

$(3)\Rightarrow(1)$
Suppose $a\in a(a^{\ast}a)^{n}R$, then for some $x_{2}\in R$ we have
$a\in a(a^{\ast}a)^{n}x_{2}=aa^{\ast}a(a^{\ast}a)^{n-1}x_{2}\in aa^{\ast}aR.$
Thus by Lemma \ref{aa-star-a}, we have $a\in R^{\dagger}$ and
$$
a^{\dagger}=[a(a^{\ast}a)^{n-1}x_{2}]^{\ast}a(a^{\ast}a)^{n-1}x_{2}a^{\ast}
=x_{2}^{\ast}(a^{\ast}a)^{n-1}a^{\ast}a(a^{\ast}a)^{n-1}x_{2}a^{\ast}
=x_{2}^{\ast}(a^{\ast}a)^{2n-1}x_{2}a^{\ast}.
$$

$(1)\Leftrightarrow(4)$ It is similar to $(1)\Leftrightarrow(3)$ and suppose $a=y_{2}(aa^{\ast})^{n}a$ for some $y_{2}\in R$,
by Lemma \ref{aa-star-a}, we have
$
a^{\dagger}=a^{\ast}y_{2}(aa^{\ast})^{n-1}a[y_{2}(aa^{\ast})^{n-1}a]^{\ast}
=a^{\ast}y_{2}(aa^{\ast})^{n-1}aa^{\ast}(aa^{\ast})^{n-1}y_{2}^{\ast}
=a^{\ast}y_{2}(aa^{\ast})^{2n-1}y_{2}^{\ast}.
$

$(1)\Rightarrow(5)$  It is easy to see that by Lemma \ref{mp-star-ep} and Lemma \ref{n-regular}.

$(1)\Rightarrow(6)$  It is similar to $(1)\Rightarrow(5)$.

$(5)\Rightarrow(4)$  Suppose $(aa^{\ast})^{n}\in R^{\dagger}$ and $(aa^{\ast})^{n}((aa^{\ast})^{n})^{\dagger}a=a$.
Let $b=(aa^{\ast})^{n}[(aa^{\ast})^{n}]^{\dagger}$, then $b^{\ast}=b$ and $ba=a$. Thus
\begin{eqnarray*}\label{mp9}
a=ba=b^{\ast}a=[(aa^{\ast})^{n}((aa^{\ast})^{n})^{\dagger}]^{\ast}a=((aa^{\ast})^{n})^{\dagger}(aa^{\ast})^{n}a\in R(aa^{\ast})^{n}a,
\end{eqnarray*}
which imply the condition $(4)$ is satisfied.

$(6)\Rightarrow(3)$  It is similar to $(5)\Rightarrow(4)$.

$(1)\Rightarrow(7)$  It is easy to see that by Lemma \ref{mp-star-ep}.

$(7)\Rightarrow(1)$  Suppose $a$ is $\ast$-cancellable and $(aa^{\ast})^{m}$ and $(a^{\ast}a)^{n}$ are regular. Then
$a^{\ast}$ is $\ast$-cancellable and $(aa^{\ast})^{m}((aa^{\ast})^{m})^{-}(aa^{\ast})^{m}=(aa^{\ast})^{m}$, thus
$(aa^{\ast})^{m}((aa^{\ast})^{m})^{-}(aa^{\ast})^{m-1}a=(aa^{\ast})^{m-1}a$. If $m-1=0$, then 
$(aa^{\ast})^{m}((aa^{\ast})^{m})^{-}a=a$, that is $a=aa^{\ast}(aa^{\ast})^{m-1}((aa^{\ast})^{m})^{-}a$, thus by Lemma
\ref{13-14-inverse}, we have $a\in R^{\{1,4\}}$. If $m-1>0$, then by $a^{\ast}$ is $\ast$-cancellable, we have
$(aa^{\ast})^{m}((aa^{\ast})^{m})^{-}(aa^{\ast})^{m-2}=(aa^{\ast})^{m-2}a$.
If $m-2=0$, then
$(aa^{\ast})^{m}((aa^{\ast})^{m})^{-}a=a$, that is $a=aa^{\ast}(aa^{\ast})^{m-1}((aa^{\ast})^{m})^{-}a$, thus by Lemma
\ref{13-14-inverse}, we have $a\in R^{\{1,4\}}$. If $m-2>0$, repeat above steps, we always have $a\in R^{\{1,4\}}$.
Similarly, by $(a^{\ast}a)^{n}$ is regular, we always have $a\in R^{\{1,3\}}$. Therefore, by Lemma \ref{mp-1314}, we have $a\in R^{\dagger}$.

$(1)\Rightarrow(8)$ By Lemma \ref{mp-star-ep}, we have $(a^{\ast}a)^{n}\in R^{EP}$ and $((a^{\ast}a)^{n})^{\dagger}=(a^{\dagger}(a^{\ast})^{\dagger})^{n}$.
Let $c=(a^{\dagger})^{\ast}((a^{\ast}a)^{\dagger})^{n}$, then
\begin{equation*}
\begin{split}
(a^{\ast}a)^{n}a^{\ast}c(a^{\ast}a)^{n}a^{\ast}
&=(a^{\ast}a)^{n}a^{\ast}(a^{\dagger})^{\ast}((a^{\ast}a)^{\dagger})^{n}(a^{\ast}a)^{n}a^{\ast}\\
&=(a^{\ast}a)^{n}[a^{\ast}(a^{\dagger})^{\ast}(a^{\ast}a)^{\dagger}]((a^{\ast}a)^{\dagger})^{n-1}(a^{\ast}a)^{n}a^{\ast}\\
&=(a^{\ast}a)^{n}[a^{\ast}(a^{\dagger})^{\ast}a^{\dagger}(a^{\ast})^{\dagger}]((a^{\ast}a)^{\dagger})^{n-1}(a^{\ast}a)^{n}a^{\ast}\\
&=(a^{\ast}a)^{n}[a^{\dagger}aa^{\dagger}(a^{\ast})^{\dagger}]((a^{\ast}a)^{\dagger})^{n-1}(a^{\ast}a)^{n}a^{\ast}\\
&=(a^{\ast}a)^{n}(a^{\ast}a)^{\dagger}((a^{\ast}a)^{\dagger})^{n-1}(a^{\ast}a)^{n}a^{\ast}\\
&=(a^{\ast}a)^{n}((a^{\ast}a)^{\dagger})^{n}(a^{\ast}a)^{n}a^{\ast}\\
&=(a^{\ast}a)^{n}((a^{\ast}a)^{n})^{\dagger}(a^{\ast}a)^{n}a^{\ast}\\
&=(a^{\ast}a)^{n}a^{\ast}.
\end{split}
\end{equation*}
Thus $(a^{\ast}a)^{n}a^{\ast}$ is regular.

$(8)\Rightarrow(7)$ Suppose $a$ is $\ast$-cancellable and $(a^{\ast}a)^{n}a^{\ast}$ is regular.
Then $$(a^{\ast}a)^{n}a^{\ast}((a^{\ast}a)^{n}a^{\ast})^{-}(a^{\ast}a)^{n}a^{\ast}=(a^{\ast}a)^{n}a^{\ast},$$ thus
$(a^{\ast}a)^{n}a^{\ast}((a^{\ast}a)^{n}a^{\ast})^{-}(a^{\ast}a)^{n}=(a^{\ast}a)^{n}$ 
and $(aa^{\ast})^{n}((a^{\ast}a)^{n}a^{\ast})^{-}a^{\ast}(aa^{\ast})^{n}=(aa^{\ast})^{n}$ 
by $a$ is $\ast$-cancellable,
that is $(a^{\ast}a)^{n}$ and $(aa^{\ast})^{n}$ are regular, therefore the condition (7) is satisfied.

$(1)\Leftrightarrow(9)$ It is similar to $(1)\Leftrightarrow(8)$.

$(1)\Rightarrow(10)$-$(13)$ It is easy to see that by Lemma \ref{mp-star-ep}.

The equivalence between $(10)$-$(13)$ can be seen by Lemma \ref{eppp}.

$(12)\Rightarrow(9)$ Suppose $a$ is $\ast$-cancellable and $(aa^{\ast})^{n}\in R^\#$, then
\begin{eqnarray}\label{group-core-dual-drazin1}
(aa^{\ast})^{n}=(aa^{\ast})^{n}[(aa^{\ast})^{n}]^\#(aa^{\ast})^{n}.
\end{eqnarray}
Pre-multiplication of (\ref{group-core-dual-drazin1}) by $a^{\ast}$ now yields
\begin{equation*}\label{group-core-dual-drazin2}
\begin{split}
&a^{\ast}(aa^{\ast})^{n}=a^{\ast}(aa^{\ast})^{n}[(aa^{\ast})^{n}]^\#(aa^{\ast})^{n}\\
&=a^{\ast}(aa^{\ast})^{n}[(aa^{\ast})^{n}]^\#[(aa^{\ast})^{n}]^\#(aa^{\ast})^{n}(aa^{\ast})^{n}\\
&=a^{\ast}(aa^{\ast})^{n}[(aa^{\ast})^{n}]^\#[(aa^{\ast})^{n}]^\#(aa^{\ast})^{n-1}a[a^{\ast}(aa^{\ast})^{n}].
\end{split}
\end{equation*}
Thus $a^{\ast}(aa^{\ast})^{n}$ is regular.
\end{proof}

\begin{definition} \cite{PRA}\label{multiple}
Let $a,b\in R$, we say that $a$ is a multiple of $b$ if $a\in Rb\cap bR$.
\end{definition}

\begin{definition} \label{leftmultiple}
Let $a,b\in R$, we say that $a$ is a left $($right$)$ multiple of $b$ if $a\in Rb$ $($$a\in bR$$)$.
\end{definition}

The existence of the Moore-Penrose inverse of an element in a ring is priori related to a Hermite element. If we take $n=1$,
the condition $(2)$ in the following theorem can be found in \cite[Theorem 1]{PRA} in the category case. 

\begin{theorem}\label{mp-in-ring-semigroup2}
Let $a\in R$ and $n\in \mathbf{N^+}$. Then the following conditions are equivalent:\\
$(1)$ $a\in R^{\dagger}$;\\
$(2)$ There exists a projection $p\in R$ such that $pa=a$ and $p$ is a multiple of $(aa^{\ast})^{n}$;\\
$(3)$ There exists a Hermite element $q\in R$ such that $qa=a$ and $q$ is a left multiple of $(aa^{\ast})^{n}$;\\
$(4)$ There exists a Hermite element $r\in R$ such that $ra=a$ and $r$ is a right multiple of $(aa^{\ast})^{n}$;\\
$(5)$ There exists $b\in R$ such that $ba=a$ and $b$ is a left multiple of $(aa^{\ast})^{n}$.
\end{theorem}
\begin{proof}
$(1)\Rightarrow(2)$ Suppose $a\in R^{\dagger}$ and let $p=aa^{\dagger}$, then $p^{2}=p=p^{\ast}$ and $pa=a$.
By Lemma \ref{normal-ep}, we have
$aa^{\ast}(a^{\dagger})^{\ast}a^{\dagger}aa^{\ast}=aa^{\ast}$,
$aa^{\ast}(a^{\dagger})^{\ast}a^{\dagger}=(a^{\dagger})^{\ast}a^{\dagger}aa^{\ast}$ and
$p=aa^{\dagger}
=(aa^{\dagger})^{\ast}=(a^{\dagger})^{\ast}a^{\ast}
=(a^{\dagger})^{\ast}(aa^{\dagger}a)^{\ast}
=(a^{\dagger})^{\ast}a^{\dagger}aa^{\ast}
=(a^{\dagger})^{\ast}a^{\dagger}aa^{\ast}(a^{\dagger})^{\ast}a^{\dagger}aa^{\ast}
=[(a^{\dagger})^{\ast}a^{\dagger}]^{2}(aa^{\ast})^{2}
=\cdots
=[(a^{\dagger})^{\ast}a^{\dagger}]^{n}(aa^{\ast})^{n}.$
By $p=p^{\ast}$, we have
$
p=p^{\ast}=[[(a^{\dagger})^{\ast}a^{\dagger}]^{n}(aa^{\ast})^{n}]^{\ast}
=(aa^{\ast})^{n}[[(a^{\dagger})^{\ast}a^{\dagger}]^{n}]^{\ast}.
$
Thus $p$ is a multiple of $(aa^{\ast})^{n}$.

$(2)\Rightarrow(3)$ It is obvious.

$(3)\Rightarrow(4)$ Let $r=q^{\ast}$.

$(4)\Rightarrow(5)$ Suppose $r^{\ast}=r$, $ra=a$ and $r$ is a right multiple of $(aa^{\ast})^{n}$, then
$
r=(aa^{\ast})^{n}w~\text{for}~\text{some}~w\in R.
$
Let $b=r$, then $ba=a$ and by $r^{\ast}=r$, we have
$
b=r=r^{\ast}=((aa^{\ast})^{n}w)^{\ast}=w^{\ast}(aa^{\ast})^{n}.
$
That is $b$ is a left multiple of $(aa^{\ast})^{n}$.

$(5)\Rightarrow(1)$ Since $b$ is a left multiple of $(aa^{\ast})^{n}$, then $b\in R(aa^{\ast})^{n}$,
post-multiplication of $b\in R(aa^{\ast})^{n}$ by $a$ now yields
$ba\in R(aa^{\ast})^{n}a$. Then by $ba=a$, which gives $a\in R(aa^{\ast})^{n}a$,
thus the condition $(4)$ in Theorem \ref{mp-in-ring-semigroup1} is satisfied.
\end{proof}

Similarly, we have the following theorem.

\begin{theorem}\label{mp-in-ring-semigroup3}
Let $a\in R$ and $n\in \mathbf{N^+}$. Then the following conditions are equivalent:\\
$(1)$ $a\in R^{\dagger}$;\\
$(2)$ There exist a projection $w\in R$ such that $aw=a$ and $w$ is a multiple of $(a^{\ast}a)^{n}$;\\
$(3)$ There exist a Hermite element $u\in R$ such that $au=a$ and $u$ is a right multiple of $(a^{\ast}a)^{n}$;\\
$(4)$ There exist a Hermite element $v\in R$ such that $av=a$ and $v$ is a left multiple of $(a^{\ast}a)^{n}$;\\
$(5)$ There exist $c\in R$ such that $ac=a$ and $c$ is a right multiple of $(a^{\ast}a)^{n}$.
\end{theorem}
If we take $n=1$, the condition $(2)$ in the following theorem can be found in \cite[Theorem 1]{PRA} in the category case.
\begin{theorem}\label{mp-in-ring-semigroup4}
Let $a\in R$ and $n\in \mathbf{N^+}$. Then the following conditions are equivalent:\\
$(1)$ $a\in R^{\dagger}$;\\
$(2)$ There exists a projection $q\in R$ such that $qa=0$ and $(aa^{\ast})^{n}+q$ is invertible;\\
$(3)$ There exists a projection $q\in R$ such that $qa=0$ and $(aa^{\ast})^{n}+q$ is left invertible;\\
$(4)$ There exists an idempotent $f\in R$ such that $fa=0$ and $(aa^{\ast})^{n}+f$ is invertible;\\
$(5)$ There exists an idempotent $f\in R$ such that $fa=0$ and $(aa^{\ast})^{n}+f$ is left invertible;\\
$(6)$ There exists $c\in R$ such that $ca=0$ and $(aa^{\ast})^{n}+c$ is invertible;\\
$(7)$ There exists $c\in R$ such that $ca=0$ and $(aa^{\ast})^{n}+c$ is left invertible.\\
In this case, \\
$a^{\dagger}=a^{\ast}y_{i}(aa^{\ast})^{2n-1}y_{i}^{\ast}$, $i\in \{1,2,3\}$,
where $1=y_{1}((aa^{\ast})^{n}+q)=y_{2}((aa^{\ast})^{n}+f)=y_{3}((aa^{\ast})^{n}+c).$
\end{theorem}
\begin{proof}
$(1)\Rightarrow(2)$ Suppose $a\in R^{\dagger}$ and let $q=1-aa^{\dagger}$, then $q^{2}=q=q^{\ast}$
and $qa=0$. By Lemma \ref{normal-ep}, we have
$
aa^{\ast}(a^{\dagger})^{\ast}a^{\dagger}=(a^{\dagger})^{\ast}a^{\dagger}aa^{\ast}.
$
Thus,
$((aa^{\ast})^{n}+q)[((a^{\dagger})^{\ast}a^{\dagger})^{n}+1-aa^{\dagger}]=1$
and $[((a^{\dagger})^{\ast}a^{\dagger})^{n}+1-aa^{\dagger}]((aa^{\ast})^{n}+q)=1.$
Therefore, $(aa^{\ast})^{n}+p$ is invertible.

$(2)\Rightarrow(3)$ It is clear.

$(3)\Rightarrow(1)$ Suppose $q^{2}=q=q^{\ast}$, $pa=0$ and $(aa^{\ast})^{n}+q$ is left invertible, then
$1=y_{1}((aa^{\ast})^{n}+q)$ for some $y_{1}\in R$. By $pa=0$, we have
$a=y_{1}((aa^{\ast})^{n}+q)a=y_{1}(aa^{\ast})^{n}a\in R(aa^{\ast})^{n}a.$
That is the condition $(4)$  in Theorem \ref{mp-in-ring-semigroup1} is satisfied and
$
a^{\dagger}=a^{\ast}y_{1}(aa^{\ast})^{n-1}a[y_{1}(aa^{\ast})^{n-1}a]^{\ast}
=a^{\ast}y_{1}(aa^{\ast})^{n-1}aa^{\ast}(aa^{\ast})^{n-1}y_{1}^{\ast}
=a^{\ast}y_{1}(aa^{\ast})^{2n-1}y_{1}^{\ast}.
$

$(1)\Rightarrow(4)$ Let $f=q=1-aa^{\dagger}$, then by $(1)\Rightarrow(2)$, which gives
$f^{2}=f\in R$, $fa=0$ and $(aa^{\ast})^{n}+f$ is invertible.

$(4)\Rightarrow(5)$ It is clear.

$(5)\Rightarrow(1)$ Suppose $f^{2}=f\in R$, $fa=0$ and $(aa^{\ast})^{n}+f$ is left invertible, then
$1=y_{2}((aa^{\ast})^{n}+f)$ for some $y_{2}\in R$. By $fa=0$, we have
$a=y_{2}((aa^{\ast})^{n}+f)a=y_{2}(aa^{\ast})^{n}a\in R(aa^{\ast})^{n}a.$
That is the condition $(4)$  in Theorem \ref{mp-in-ring-semigroup1} is satisfied and
$
a^{\dagger}=a^{\ast}y_{2}(aa^{\ast})^{n-1}a[y_{2}(aa^{\ast})^{n-1}a]^{\ast}
=a^{\ast}y_{2}(aa^{\ast})^{n-1}aa^{\ast}(aa^{\ast})^{n-1}y_{2}^{\ast}
=a^{\ast}y_{2}(aa^{\ast})^{2n-1}y_{2}^{\ast}.
$

$(1)\Rightarrow(6)$ Let $c=q=1-aa^{\dagger}$, then by $(1)\Rightarrow(2)$, which gives
$ca=0$ and $(aa^{\ast})^{n}+c$ is invertible. Since $c=q$ and $q^{2}=q=q^{\ast}$, thus
$(aa^{\ast})^{n}+q$ is invertible implies $(aa^{\ast})^{n}+c$ is invertible.

$(6)\Rightarrow(7)$ It is clear.

$(7)\Rightarrow(1)$ Suppose $ca=0$ and $(aa^{\ast})^{n}+c$ is left invertible, then
$1=y_{3}((aa^{\ast})^{n}+c)$ for some $y_{3}\in R$. By $ca=0$, we have
$a=y_{3}((aa^{\ast})^{n}+c)a=y_{3}(aa^{\ast})^{n}a\in R(aa^{\ast})^{n}a.$
That is the condition $(4)$ in Theorem \ref{mp-in-ring-semigroup1} is satisfied and
\begin{equation*}
a^{\dagger}=a^{\ast}y_{3}(aa^{\ast})^{n-1}a[y_{3}(aa^{\ast})^{n-1}a]^{\ast}
=a^{\ast}y_{3}(aa^{\ast})^{n-1}aa^{\ast}(aa^{\ast})^{n-1}y_{3}^{\ast}
=a^{\ast}y_{3}(aa^{\ast})^{2n-1}y_{3}^{\ast}.
\end{equation*}
\end{proof}
Similarly, we have the following theorem.
\begin{theorem}\label{mp-in-ring-semigroup5}
Let $a\in R$ and $n\in \mathbf{N^+}$. Then the following conditions are equivalent:\\
$(1)$ $a\in R^{\dagger}$;\\
$(2)$ There exists a projection $p\in R$ such that $ap=0$ and $(a^{\ast}a)^{n}+p$ is invertible;\\
$(3)$ There exists a projection $p\in R$ such that $ap=0$ and $(a^{\ast}a)^{n}+p$ is right invertible;\\
$(4)$ There exists an idempotent $e\in R$ such that $ae=0$ and $(a^{\ast}a)^{n}+e$ is invertible;\\
$(5)$ There exists an idempotent $e\in R$ such that $ae=0$ and $(a^{\ast}a)^{n}+e$ is right invertible;\\
$(6)$ There exists $b\in R$ such that $ab=0$ and $(a^{\ast}a)^{n}+b$ is invertible;\\
$(7)$ There exists $b\in R$ such that $ab=0$ and $(a^{\ast}a)^{n}+b$ is right invertible.\\
In this case, \\
$a^{\dagger}=x_{i}^{\ast}(a^{\ast}a)^{2n-1}x_{i}a^{\ast}$, $i\in \{1,2,3\}$,
where $1=((aa^{\ast})^{n}+p)x_{1}=((aa^{\ast})^{n}+e)x_{2}=((aa^{\ast})^{n}+b)x_{3}$.
\end{theorem}

If we take $n=1$ in the equivalent condition $(2)$ in Theorem \ref{mp-in-ring-semigroup5}, we can get the condition
$a$ is left $\ast$-cancellable in \cite[Theorem 1]{KDC} can be dropped.
An element $a\in R$ is called co-supported if there exists a projection $q\in R$ such that $qa=a$ and $aa^{\ast}+1-q$ is invertible.
In \cite{KDC}, Koliha, Djordjevi\'{c} and Cvetkvi\'{c}  also proved that $a\in R^{\dagger}$ if and only if $a$ is right $\ast$-cancellable and co-supported.
If we take $n=1$ in the equivalent condition $(2)$ in Theorem \ref{mp-in-ring-semigroup4}, we can get the condition
$a$ is right $\ast$-cancellable can be dropped. Thus we have the following corollary.
\begin{corollary} \label{cora}
Let $a\in R$. Then the following conditions are equivalent:\\
$(1)$ $a\in R^{\dagger}$;\\
$(2)$ $a$ is well-supported;\\
$(3)$ $a$ is co-supported.
\end{corollary}

\begin{lemma} \cite{J}\label{jacobson-lemma}
Let $a,b\in R$. Then we have:\\
$(1)$ $1-ab$ is left invertible if and only if $1-ba$ is left invertible;\\
$(2)$ $1-ab$ is right invertible if and only if $1-ba$ is right invertible;\\
$(3)$ $1-ab$ is invertible if and only if $1-ba$ is invertible.
\end{lemma}

In \cite{PP1}, Patri\'{c}io proved that if $a\in R$ is regualr with $a^{-}\in a\{1\}$, then
$a\in R^{\dagger}$ if and only if $a^{\ast}a+1-a^{-}a$ is invertible if and only if $aa^{\ast}+1-aa^{-}$ is invertible.
In the following theorem, we show that $a^{\ast}a$ and $aa^{\ast}$ can be replace by $(a^{\ast}a)^{n}$ and $(aa^{\ast})^{n}$, respectively.
Moreover, the invertibility can be generalized to the one sided invertibility.

\begin{theorem} \label{corb}
Let $a\in R$ and $n\in \mathbf{N^+}$. If $a$ is regular and $a^{-}\in a\{1\}$, then the following conditions are equivalent:\\
$(1)$ $a\in R^{\dagger}$;\\
$(2)$ $v=(a^{\ast}a)^{n}+1-a^{-}a$ is invertible;\\
$(3)$ $v=(a^{\ast}a)^{n}+1-a^{-}a$ is right invertible;\\
$(4)$ $u=(aa^{\ast})^{n}+1-aa^{-}$ is invertible;\\
$(5)$ $u=(aa^{\ast})^{n}+1-aa^{-}$ is left invertible.
\end{theorem}
\begin{proof}
$(2)\Leftrightarrow(4)$
By $u=(aa^{\ast})^{n}+1-aa^{-}=1+a[a^{\ast}(aa^{\ast})^{n-1}-a^{-}]$ is invertible and Lemma \ref{jacobson-lemma}, we have $u$ is invertible is equivalent to
$1+[a^{\ast}(aa^{\ast})^{n-1}-a^{-}]a=(a^{\ast}a)^{n}+1-a^{-}a$ is invertible, which is $v$ is invertible.

$(1)\Rightarrow(2)$ If $a\in R^{\dagger}$, then by Lemma \ref{aa-star-a} we have
\begin{eqnarray}\label{mary-regular-unitw}
a=x^{\ast}aa^{\ast}a~\text{and}~a=aa^{\ast}ay^{\ast}~\text{for}~\text{some}~x,y\in R.
\end{eqnarray}
Taking involution on (\ref{mary-regular-unitw}), we have
$
a^{\ast}=a^{\ast}aa^{\ast}x~\text{and}~a^{\ast}=ya^{\ast}aa^{\ast}.
$
Thus
\begin{eqnarray}\label{mary-regular-unit2}
a^{\ast}=a^{\ast}aa^{\ast}x=a^{\ast}a(a^{\ast}aa^{\ast}x)x=(a^{\ast}a)^{2}a^{\ast}x^{2}=\cdots\cdots=(a^{\ast}a)^{n}a^{\ast}x^{n}.\\
\label{mary-regular-unit3}
a^{\ast}=ya^{\ast}aa^{\ast}=y(ya^{\ast}aa^{\ast})aa^{\ast}=y^{2}a^{\ast}(aa^{\ast})^{2}=\cdots\cdots=y^{n}a^{\ast}(aa^{\ast})^{n}.
\end{eqnarray}
By (\ref{mary-regular-unit2}) and (\ref{mary-regular-unit3}), we have
\begin{equation} \label{mary-regular-unit4}
\begin{split}
&[(a^{\ast}a)^{n}a^{\ast}(a^{-})^{\ast}+1-a^{\ast}(a^{-})^{\ast}](a^{\ast}x^{n}(a^{-})^{\ast}+1-a^{\ast}(a^{-})^{\ast})\\
&=(a^{\ast}a)^{n}a^{\ast}(a^{-})^{\ast}a^{\ast}x^{n}(a^{-})^{\ast}+(a^{\ast}a)^{n}a^{\ast}(a^{-})^{\ast}(1-a^{\ast}(a^{-})^{\ast})\\
&+(1-a^{\ast}(a^{-})^{\ast})a^{\ast}x^{n}(a^{-})^{\ast}+(1-a^{\ast}(a^{-})^{\ast})^{2}\\
&=(a^{\ast}a)^{n}a^{\ast}(a^{-})^{\ast}a^{\ast}x^{n}(a^{-})^{\ast}+(1-a^{\ast}(a^{-})^{\ast})^{2}\\
&=(a^{\ast}a)^{n}a^{\ast}x^{n}(a^{-})^{\ast}+1-a^{\ast}(a^{-})^{\ast}\\
&=a^{\ast}(a^{-})^{\ast}+1-a^{\ast}(a^{-})^{\ast}\\
&=1.
\end{split}
\end{equation}
and
\begin{equation} \label{mary-regular-unit5}
\begin{split}
&(y^{n}a^{\ast}(a^{-})^{\ast}+1-a^{\ast}(a^{-})^{\ast})[(a^{\ast}a)^{n}a^{\ast}(a^{-})^{\ast}+1-a^{\ast}(a^{-})^{\ast}]\\
&=y^{n}a^{\ast}(a^{-})^{\ast}(a^{\ast}a)^{n}a^{\ast}(a^{-})^{\ast}+y^{n}a^{\ast}(a^{-})^{\ast}(1-a^{\ast}(a^{-})^{\ast})\\
&+(1-a^{\ast}(a^{-})^{\ast})(a^{\ast}a)^{n}a^{\ast}(a^{-})^{\ast}+(1-a^{\ast}(a^{-})^{\ast})^{2}\\
&=y^{n}a^{\ast}(a^{-})^{\ast}(a^{\ast}a)^{n}a^{\ast}(a^{-})^{\ast}+(1-a^{\ast}(a^{-})^{\ast})^{2}\\
&=y^{n}a^{\ast}(a^{-})^{\ast}a^{\ast}(aa^{\ast})^{n}(a^{-})^{\ast}+1-a^{\ast}(a^{-})^{\ast}\\
&=a^{\ast}(a^{-})^{\ast}+1-a^{\ast}(a^{-})^{\ast}\\
&=1.
\end{split}
\end{equation}
By (\ref{mary-regular-unit4}) and (\ref{mary-regular-unit5}) we have
$(a^{\ast}a)^{n}a^{\ast}(a^{-})^{\ast}+1-a^{\ast}(a^{-})^{\ast}=1+[(a^{\ast}a)^{n}-1]a^{\ast}(a^{-})^{\ast}$ is invertible.
By Lemma \ref{jacobson-lemma}, we have
$1+a^{\ast}(a^{-})^{\ast}[(a^{\ast}a)^{n}-1]=(a^{\ast}a)^{n}+1-a^{\ast}(a^{-})^{\ast}$ is invertible, which is $v^{\ast}$ is invertible
, thus $v$ is invertible.

$(2)\Rightarrow(3)$ and $(4)\Rightarrow(5)$ It is clear.

$(3)\Rightarrow(1)$ and $(5)\Rightarrow(1)$ It is easy to see that by Theorem \ref{mp-in-ring-semigroup5} and Theorem \ref{mp-in-ring-semigroup4}, respectively.
\end{proof}

In \cite{H}, Hartwig proved that $a\in R^{\{1,3\}}$ if and only if $R=aR\oplus (a^{\ast})^{\circ}$. And, it is also proved that
$a\in R^{\{1,4\}}$ if and only if $R=Ra\oplus ^{\circ}\!(a^{\ast})$.
Hence $a\in R^{\dagger}$ if and only if $R=aR\oplus (a^{\ast})^{\circ}=Ra\oplus ^{\circ}\!(a^{\ast})$.

\begin{theorem}\label{mp-in-ring-semigroup6}
Let $a\in R$ and $n\in \mathbf{N^+}$. Then the following conditions are equivalent:\\
$(1)$ $a\in R^{\dagger}$;\\
$(2)$ $R=a^{\circ}\oplus (a^{\ast}a)^{n}R$;\\
$(3)$ $R=a^{\circ}+ (a^{\ast}a)^{n}R$;\\
$(4)$ $R=(a^{\ast})^{\circ}\oplus (aa^{\ast})^{n}R$;\\
$(5)$ $R=(a^{\ast})^{\circ}+ (aa^{\ast})^{n}R$;\\
$(6)$ $R=^{\circ}\!a\oplus R(aa^{\ast})^{n}$;\\
$(7)$ $R=^{\circ}\!a+ R(aa^{\ast})^{n}$;\\
$(8)$ $R=^{\circ}\!(a^{\ast})\oplus R(a^{\ast}a)^{n}$;\\
$(9)$ $R=^{\circ}\!(a^{\ast})+ R(a^{\ast}a)^{n}$.
\end{theorem}
\begin{proof}
$(1)\Rightarrow(2)$
Suppose $a\in R^{\dagger}$, then by Theorem \ref{mp-in-ring-semigroup1} we have $a\in a(a^{\ast}a)^{n}R$, that is
$
a=a(a^{\ast}a)^{n}b~ \text{for}~\text{some}~b\in R.
$
Thus $a[1-(a^{\ast}a)^{n}b]=0$, which is equivalent to $1-(a^{\ast}a)^{n}b\in a^{\circ}$.

By $1=1-(a^{\ast}a)^{n}b+(a^{\ast}a)^{n}b\in a^{\circ}+(a^{\ast}a)^{n}R$, we have
$
R=a^{\circ}+(a^{\ast}a)^{n}R.
$
Let $u\in a^{\circ}\cap (a^{\ast}a)^{n}R$, then we have
$
au=0~\text{and}~u=(a^{\ast}a)^{n}v,~\text{for}~\text{some}~v\in R.
$
Hence
$
u
=(a^{\ast}a)^{n}v=a^{\ast}a(a^{\ast}a)^{n-1}v
=(a(a^{\ast}a)^{n}b)^{\ast}a(a^{\ast}a)^{n-1}v
=b^{\ast}(a^{\ast}a)^{n}a^{\ast}a(a^{\ast}a)^{n-1}v
=b^{\ast}(a^{\ast}a)^{n}(a^{\ast}a)^{n}v
=b^{\ast}(a^{\ast}a)^{n}u
=b^{\ast}(a^{\ast}a)^{n-1}a^{\ast}(au)
=0.
$
Whence $R=a^{\circ}\oplus (a^{\ast}a)^{n}R$.

$(2)\Rightarrow(3)$ It is clear.

$(3)\Rightarrow(1)$ Suppose $R=a^{\circ}+ (a^{\ast}a)^{n}R$, Pre-multiplication of $R=a^{\circ}+ (a^{\ast}a)^{n}R$ by $a$ now yields
$
aR=aa^{\circ}+ a(a^{\ast}a)^{n}R.
$
By $aa^{\circ}=0$, we have $a\in a(a^{\ast}a)^{n}R$, that is the condition $(3)$ in Theorem \ref{mp-in-ring-semigroup1} is satisfied.

By the equivalence between $(1)$, $(2)$ and $(3)$ and Lemma \ref{star-mp}, which implies
the equivalence between $(1)$, $(4)$ and $(5)$. The equivalence between $(1)$, $(6)$-$(9)$
is similar to the equivalence between $(1)$, $(2)$-$(5)$.
\end{proof}

If we take $n=1$ in the equivalent conditions $(2)$-$(9)$ in Theorem \ref{mp-in-ring-semigroup6}, we have the following corollary.
\begin{corollary} \label{corc}
Let $a\in R$. The following conditions are equivalent:\\
$(1)$ $a\in R^{\dagger}$;\\
$(2)$ $R=a^{\circ}\oplus a^{\ast}aR$;\\
$(3)$ $R=a^{\circ}+ a^{\ast}aR$;\\
$(4)$ $R=(a^{\ast})^{\circ}\oplus aa^{\ast}R$;\\
$(5)$ $R=(a^{\ast})^{\circ}+ aa^{\ast}R$;\\
$(6)$ $R=^{\circ}\!a\oplus Raa^{\ast}$;\\
$(7)$ $R=^{\circ}\!a+ Raa^{\ast}$;\\
$(8)$ $R=^{\circ}\!(a^{\ast})\oplus Ra^{\ast}a$;\\
$(9)$ $R=^{\circ}\!(a^{\ast})+ Ra^{\ast}a$.
\end{corollary}

\centerline {\bf ACKNOWLEDGMENTS} This research is supported by the National Natural Science Foundation of China (No. 11201063 and No. 11371089), the Specialized Research Fund for the Doctoral Program of Higher Education (No. 20120092110020); the Jiangsu Planned Projects for Postdoctoral Research Funds (No. 1501048B); the Natural Science Foundation of Jiangsu Province (No. BK20141327); the Foundation of Graduate Innovation Program of Jiangsu Province (No. KYZZ15-0049).

\end{document}